\documentclass[11pt,notitlepage,twoside,a4paper]{amsart}
 \usepackage{amsfonts}

\usepackage{amsmath,amssymb,enumerate}

\usepackage{epsfig}%,fancyhdr,color}%,showkeys,amsmidx

\usepackage{amssymb}
\usepackage{amsmath,amsthm}
\usepackage{latexsym}
\usepackage{amscd}
\usepackage{psfrag}
\usepackage{graphicx}
\usepackage[latin1]{inputenc}
\usepackage[mathcal]{eucal}

\renewcommand{\textsc}{\textcolor{red}}

\newtheorem{theorem}{\rm\bf Theorem}[section]

\newtheorem{lemma}[theorem]{\rm\bf Lemma}
\newtheorem{corollary}[theorem]{\rm\bf Corollary}
\newtheorem*{theorem 1}{\rm\bf Proposition 1}
\newtheorem*{theorem 2}{\rm\bf Proposition 2}

\theoremstyle{definition}

\theoremstyle{remark}

\def\interieur#1{\mathord{\mathop{\kern 0pt #1}\limits^\circ}}

\title[The arc and curve complex]{On the arc and curve complex of a surface}

\author{Mustafa Korkmaz}
\address{Mustafa Korkmaz, Department of Mathematics, Middle East Technical University, 06531 Ankara, Turkey.}
\email{korkmaz@metu.edu.tr}

\author{Athanase Papadopoulos}
\address{Athanase Papadopoulos,  Institut de Recherche Math{\'e}matique Avanc\'ee,
Universit{\'e} de Strasbourg and CNRS,
7 rue Ren\'e Descartes,
 67084 Strasbourg Cedex, France,  and
  Max-Planck-Institut f\"ur Mathematik, Vivatsgasse 7, 53111 Bonn, Germany.} \email{papadopoulos@math.u-strasbg.fr}

\date{\today}

\begin{document}

\begin{abstract}
We study the {\it arc and curve} complex $AC(S)$ of an oriented
connected surface $S$ of finite type with punctures. We show that
if the surface is not  a sphere with one, two or three punctures
nor a torus with one puncture, then the simplicial automorphism
group of $AC(S)$ coincides with the natural image of the extended
mapping class group of $S$ in that group. We also show that for
any vertex of $AC(S)$, the combinatorial structure of the link of
that vertex  characterizes the type of a curve or of
an arc  in $S$ that represents that vertex. We also give a proof
of the fact if $S$ is not a sphere with at most three punctures, then the natural embedding of the curve
complex of $S$ in $AC(S)$  is a quasi-isometry. The last result, at least under some slightly more restrictive conditions on $S$, was
already known. As a corollary, $AC(S)$ is Gromov-hyperbolic.
 \bigskip


\medskip


\end{abstract}
\maketitle

\section{Introduction}
\label{intro}

In this paper, $S=S_{g,n}$ is a connected orientable surface of genus $g\geq 0$
without boundary and with $n\geq 0$ punctures.  The {\it mapping class group} of
$S$,
denoted by $\mathrm{Mod}(S)$, is the group of isotopy classes of
orientation-preserving
homeomorphisms of $S$.  The {\it extended mapping class group} of $S$,
$\mathrm{Mod}^*(S)$, is the group of isotopy classes
of all homeomorphisms of $S$.

We shall denote by $\overline{S}$ the closed surface obtained by filling in
the
punctures of $S$, and by $B$ the set of punctures when these points are
considered as elements of $\overline{S}$. The elements of $B$
are also called
{\it distinguished points} of  $\overline{S}$.

A simple closed curve on $S$ or on $\overline{S}$ is said to be {\it essential}
if it is  not homotopic to a point or to a puncture (or a distinguished point).

An {\it arc} in $S$ or in  $\overline{S}$  is the homeomorphic image of a closed
interval in $\overline{S}$
whose interior is in  $\overline{S}\setminus B$ and whose endpoints are in $B$.
An arc is said to be {\it essential} if it is not homotopic (relative to $B$) to
a point in $\overline{S}$.

 When the setting is clear, we shall use the term ``curve"
 instead of ``essential simple closed curve" and the term ``arc" instead of
 ``essential arc".

The {\it curve complex} of $S$, denoted by $C(S)$, is the abstract simplicial complex
whose $k$-simplices, for each $k\geq 0$, are the sets of $k+1$ distinct isotopy classes
of curves in $S$ that can be represented by pairwise disjoint
curves on this surface. The curve complex has been introduced by Harvey \cite{harvey},
and since then it has been investigated from various points of view (see in particular
the papers \cite{ivanov}, \cite{korkmaz}, \cite{luo}  and \cite{MM}).
We shall use below an alternative definition of $C(S_{g,n})$ in the case $(g,n)=(0,4)$ or $(g,n)=(1,1)$.

The {\it arc complex} of $S$, denoted by $A(S)$, is defined analogously, with  curves
replaced by arcs. The arc complex has also been studied by several authors,
see for instance \cite{Hatcher}, \cite{IM}, \cite{ivanov} and \cite{korkmaz}.

In this paper, we study the  {\it arc and curve complex}, $AC(S)$,
an abstract simplicial complex in which the curve complex and the arc complex
naturally embed.
 The arc and curve complex is the abstract simplicial complex whose $k$-simplices,
 for each $k\geq 0$, are collections of $k+1$ distinct isotopy classes of
 one-dimensional submanifolds which can be either essential simple closed
 curves or essential arcs in $S$, such that this collection can be
 represented by disjoint curves or arcs on the surface. The arc and curve complex
 was studied by Hatcher in \cite{Hatcher}, who proved that this complex is
 contractible.

Note that if $n=0$, then there are no arcs on
$S$, and in that case $AC(S_{g,0})=C(S_{g,0})$.  From now on, we assume that
$n\geq 1$.

Each element of the extended mapping class group $\mathrm{Mod}^*(S)$
acts naturally in a simplicial way on the complex $AC(S)$, and  its is clear
that the resulting map from $\mathrm{Mod}^*(S)$ to the
simplicial automorphism group  $\mathrm{Aut}(AC(S))$ of $AC(S)$ is a homomorphism.

There are natural simplicial maps from the curve complex $C(S)$ and
from the arc complex $A(S)$ into
the arc and curve complex, which extend the natural inclusions at the
level of the vertices.

In this paper, we prove the following:

\begin{theorem}\label{th:auto}
If the surface $S_{g,n}$ is not a sphere with one, two or
three punctures nor a torus with one puncture, then
the natural homomorphism $\mathrm{Mod}^*(S_{g,n})\to \mathrm{Aut}(AC(S_{g,n}))$
is an isomorphism.
\end{theorem}

The analogous result for the curve complex is due to Ivanov, Korkmaz and Luo
(\cite{ivanov}, \cite{korkmaz}, and \cite{luo}), Theorem \ref{thm:complex} below.
The analogous
result for the arc complex has been recently obtained by Irmak and
McCarthy~(\cite{IM}), Theorem~\ref{thm:IM} below.

We shall give two different proofs of Theorem~\ref{th:auto}. The first one involves
passing to the arc complex, and it uses the result by Irmak and McCarthy.
The second proof involves passing to the
curve complex, and it is based on the result of  Ivanov, Korkmaz and Luo.
We note that the proof of Irmak and McCarthy's result on the
automorphisms of the arc complex does
not use the result by  Ivanov, Korkmaz and Luo.

We then show the following:

\begin{theorem}\label{th:distinct}
For any surface $S$, the combinatorics of the link at each vertex of $AC(S)$ is
sufficient to characterize the type of curve or arc that represents such a vertex.
 \end{theorem}

We shall be more precise about the meaning of
Theorem~\ref{th:distinct}. Note that it is a  local result,
which does not immediately follow from Theorem~\ref{th:auto}.

If $(g,n)=(0,4)$ or $(g,n)=(1,1)$, the curve complex $C(S)$ is infinite discrete in the above
definition. For Theorem~\ref{th:quasi} below, we change the definition of
$C(S)$ in these two cases as follows:
The complex $C(S)$ is the graph whose vertices are isotopy classes of simple closed
curves and an edge is placed between two vertices if two representatives of the
vertices have minimal intersection number, which is $1$ if $g=1$ and $2$ if $g=0$.

Here and throughout the paper, we shall consider geometric realization of the simplicial complexes $AC(S)$ and $C(S)$, and not only these complexes as abstract simplicial complexes. In particular, we make these complexes into complete geodesic
metric spaces in a natural way by making each simplex a regular
Euclidean simplex of sidelength $1$, as in \cite{MM}.
We call these metrics on  $AC(S)$ and $A(S)$ the {\it simplicial metrics} on these spaces.
In the last section of the paper, we give a proof of the following result:

\begin{theorem}\label{th:quasi}
If the surface $S_{g,n}$ is not a sphere with at most three punctures,
then the natural map
$C(S_{g,n})\to AC(S_{g,n})$ is a quasi-isometry. Furthermore, each point in
$ AC(S_{g,n})$ is at distance at most one from a point in the
image of $C(S_{g,n})$.
\end{theorem}

For the cases where $(g,n)=(0,4)$ or $(1,1)$, a special discussion is needed. The result given
in Theorem~\ref{th:quasi}, with slighty less general conditions on the topological type of the
surface $S_{g,n}$,
is not new; it is mentioned in Schleimer's notes \cite{Schleimer}. Although
the result is independent from the rest of the paper, we think it is useful
to provide a complete proof of it.

The complex $AC(S)$ is finite if $S$ is a sphere with at most three punctures.
As a consequence of Theorem~\ref{th:quasi}, and the fact that the curve complex is
Gromov-hyperbolic (see \cite{MM}), we have the following corollary:

\begin{corollary} The arc and curve complex $AC(S_{g,n})$ of any  surface  $S_{g,n}$ is Gromov-hyperbolic.
\end{corollary}
 
 Note that the condition that  $S_{g,n}$ is not a surface with at most three punctures is not excluded in this corollary, because finite metric spaces are always Gromov-hyperbolic.

\section{Maximal simplices in $AC(S)$}

The results of this section will be used in the proof of Theorem  \ref{th:auto}.

A simplex in $AC(S)$ is said to be {\it maximal} if it is maximal in the sense of
inclusion.

Let $\Delta$ be a maximal simplex in $AC(S)$, and suppose that the vertices of $\Delta$
are isotopy classes of arcs.
Then a set of disjoint representatives of $\Delta$ is an {\it  ideal
triangulation} of the closed surface $\overline{S}$, that is, a triangulation having all
its vertices at the distinguished points. We shall also call $\Delta$ an {\it  ideal
triangulation} of $S$.

\begin{figure}[hbt]
 \begin{center}
    \includegraphics[width=12cm]{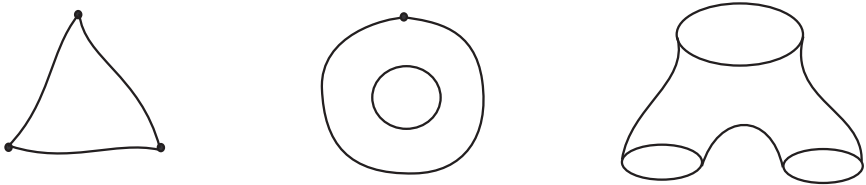}
  \caption{Possible components of the complement of a maximal simplex in the complex
  $AC(S)$.}
  \label{complement}
   \end{center}
 \end{figure}

Let us take again a maximal simplex $\Delta$ in $AC(S)$ and let
$\delta$ be a set of representatives of elements of $\Delta$ that
are pairwise disjoint. Each component of the surface $S_\delta$
obtained by cutting $S$ along $\delta$ is of one of the following
three forms: (1) a triangle; (2) an annulus with one of its
boundary components being a simple closed curve in $\delta$ and
the other boundary component being an arc joining a puncture to
itself; (3) a pair of pants all of whose boundary components are
simple closed curves (see Figure~\ref{complement}).

We shall use the following:

\begin{lemma} \label{lemma:max}
A maximal simplex of $\Delta$ in $AC(S)$ that has maximal dimension consists of
arcs, i.e. it is an ideal triangulation of $S$. The dimension of such a simplex is
$6g+3n-7$.
\end{lemma}

 \begin{figure}[hbt]
 \begin{center}
    \includegraphics[width=8cm]{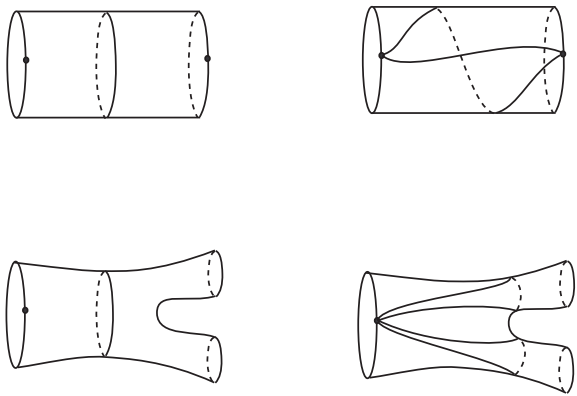}
  \caption{The dimension of the maximal simplex is increased in the presence of
  a curve.}
  \label{increase}
   \end{center}
 \end{figure}

 \begin{proof}  The first statement is a consequence of the fact that
 if we start with a maximal simplex in $AC(S)$
having at least one vertex that is the isotopy class of a simple closed curve,
 then we can obtain a maximal simplex of one dimension higher by replacing
 one of these simple closed curves by two arcs (see Figure~\ref{increase}).
 Thus, a maximal simplex $\Delta$ having maximal dimension
 consists only of isotopy classes of arcs,
 and it is therefore the isotopy class of an ideal triangulation of $S$.
 The dimension statement follows then by a standard Euler characteristic
 argument, which gives the number of curves in an ideal triangulation of $S$.
 This number is $6g+3n-6=-3\chi(S)$, which implies that the dimension of
 the maximal simplex $\Delta$ is $6g+3n-7$.
 \end{proof}

 \begin{lemma} \label{lemma:min}
The dimension of a maximal simplex $\Delta$ in $AC(S)$ that has minimal dimension
is $3g+2n-4$.
\end{lemma}

 \begin{proof}
 The proof is analogous to the proof of Lemma~\ref{lemma:max}.
We use the fact that a maximal simplex that has minimal dimension has a maximal
number of curves,
which is $3g+n-3$. If we cut the surface $S$ along such a collection of curves,
we get a collection of subsurfaces of Euler characteristic $-1$. Each of these
subsurfaces
is either a disc with two punctures, or an annulus with one puncture or a pair of
pants.
On a disc with two punctures, there are systems of two disjoint nonisotopic arcs
(and no systems of three nonisotopic arcs);
on an annulus with one puncture there is only one essential arc;  on a pair of
pants there are no arcs. Hence,
the number of curves and arcs on $S$ forming a maximal simplex of minimal dimension
is $3g+n-3+n= 3g+2n-3$.
The dimension of that simplex is therefore $3g+2n-4$.
 \end{proof}

From Lemmas  \ref{lemma:max} and  \ref{lemma:min}, it follows that the dimension of
maximal simplices in $AC(S)$ is bounded between $3g+2n-4$ and $6g+3n-7$.
 It can be seen, using the arguments of the proof of these lemmas that every
 integer between these two bounds is realized as the dimension of a maximal simplex.

\section{Proof of Theorem \ref{th:auto}}

We now give our first proof of Theorem \ref{th:auto}.
We may assume that the arc complex of the surface is not empty.
Recall that we have the assumption $n\geq 1$. We use the following:

\begin{theorem} [Irmak-McCarthy \cite{IM}] \label{thm:IM}
Let $S$ be a connected orientable
surface of genus $g\geq 0$ with $n\geq 1$ punctures,
which is not homeomorphic to a sphere with one, two or three punctures,
or a torus with one puncture. Then the natural map
$\mathrm{Mod}^*(S)\to \mathrm{Aut}(A(S))$ is an isomorphism. Furthermore, in the
excluded cases, the natural map $\mathrm{Mod}^*(S)\to \mathrm{Aut}(A(S))$ is
surjective and its kernel is the center of $\mathrm{Mod}^*(S)$.
 \qed
\end{theorem}

For the proof of Theorem~\ref{th:auto}, let $\Phi
:\mathrm{Mod}^*(S)\to \mathrm{Aut}(AC(S))$ denote the natural map.
We define a natural map $\mathrm{Aut}(AC(S))\to
\mathrm{Aut}(A(S))$ as follows: Let $\varphi$ be an automorphism
of $AC(S)$. If $a$ is an arc on $S$, then $a$ is contained in a
maximal simplex $\Delta$ of dimension $6g+3n-7$. (Here and in the
sequel, we do not distinguish between an arc and its isotopy
class. Likewise, we do not distinguish between a curve and its
isotopy class.) Since $\varphi:AC(S)\to AC(S)$ is a simplicial
automorphism, the image simplex $\varphi(\Delta)$ is also of
dimension $6g+3n-7$ containing $\varphi(a)$. By
Lemma~\ref{lemma:max}, $\varphi(a)$ is an arc. It follows that
$\varphi$ maps arcs to arcs. Since the inverse of $\varphi$ also
maps arcs to arcs, the restriction of $\varphi$ to arcs gives an
automorphism  $\tilde {\varphi}:A(S)\to A(S)$. It is easy to see
that the resulting map  $\varphi \mapsto \tilde {\varphi}$ is a
homomorphism, which we denote by $\Psi : \rm{Aut}(AC(S)) \to
\rm{Aut}(A(S))$.

Let $\Theta :\mathrm{Mod}^*(S)\to
\mathrm{Aut}(A(S))$ be the natural map, which is an isomorphism by
Theorem~\ref{thm:IM}. Clearly, we have the equality $\Psi\circ  \Phi =\Theta$,
deducing that $\Phi$ is one-to-one and $\Psi$ is onto.
\begin{center}
\unitlength=1mm
\begin{picture}(70,20)(0,5)
\put(20,22){\makebox(0,0){$\mathrm{Mod}^*(S)$}}
\put(20,6){\makebox(0,0){$\mathrm{Aut}(AC(S))$}}
\put(58,13){\makebox(0,0){$\mathrm{Aut}(A(S))$}}
\put(20,19){\vector(0,-1){10}}
\put(30,21){\vector(3,-1){17}}
\put(32,6){\vector(3,1){15}}
\put(23,14){\makebox(0,0){$\Phi$}}
\put(40,11){\makebox(0,0){$\Psi$}}
\put(40,21){\makebox(0,0){$\Theta$}}
\end{picture}
\end{center}

Let $\varphi$ be an automorphism of $AC(S)$ such that $\varphi$
fixes each arc. We claim that $\varphi$ is the identity. This will show that
$\Psi$ is one-to-one, concluding the proof that the natural map $\Phi$
is onto.

Let $b$ be a curve on $S$. It suffices to show that
$\varphi (b)=b$. Suppose first that $b$ is nonseparating. It is easy to see that
there is a maximal simplex $\Delta$ in $AC(S)$
containing $b$ such that all elements of
$\Delta'=\Delta-\{ b\}$ are arcs and that $b$ is the only curve
disjoint from all elements of $\Delta'$.
Since the automorphism $\varphi$ fixes each element
$\Delta'$ and $\varphi (b)$ is a curve disjoint from
all elements of $\Delta'$, we have $\varphi (b)=b$.
If $c$ is a separating curve on $S$, then it is easy to find
a maximal simplex $\Sigma$ in $AC(S)$ containing $c$ such that
each element of $\Sigma'=\Sigma-\{c\}$ is either a nonseparating curve or an arc,
and that $c$ is the only curve disjoint from all elements of $\Sigma'$.
Since $\varphi$ fixes each element $\Sigma'$, it must fix $c$ as well.

This completes the first proof of Theorem \ref{th:auto}.

\bigskip

We now give our second proof of Theorem \ref{th:auto}, using the curve
complex instead of the arc complex. We use the following:

\begin{theorem} [Ivanov-Korkmaz-Luo \cite{ivanov,korkmaz,luo}] \label{thm:complex}
Let $S$ be a surface of genus $g$ with $n$ punctures. If $2g+n\geq
5$ then the natural map $\mathrm{Mod}^*(S)\to \mathrm{Aut}(C(S))$ is an
isomorphism.
 \qed
\end{theorem}

Note that the hypothesis (and therefore the second proof of Theorem \ref{th:auto}
that we give below) excludes  the cases where the surface $S$ is the sphere with
four punctures or the torus with two punctures, but these cases were taken care in
the first proof.

For the second proof of Theorem~\ref{th:auto}, let $\Phi:
\mathrm{Mod}^*(S)\to \mathrm{Aut}(AC(S))$ be the natural map, as above.
If $\varphi$ is an automorphism of $AC(S)$, it maps arcs to arcs and
curves to curves, as shown in the first proof above.
Hence, $\varphi$ induces an automorphism
$\tilde{\varphi} :C(S) \to C(S)$, giving a homomorphism
$\Psi: \mathrm{Aut}(AC(S))\to \mathrm{Aut}(C(S))$ defined by
$\varphi\mapsto \tilde{\varphi}$.

Let $\Theta:\mathrm{Mod}^*(S)\to \mathrm{Aut}(C(S))$ be the natural homomorphism,
which is an isomorphism by Theorem~\ref{thm:complex}. It is easy to see that
$\Psi \circ \Phi =\Theta$, deducing that $\Phi$ is one-to-one and $\Psi$ is onto.

\begin{center}
\unitlength=1mm
\begin{picture}(70,20)(0,5)
\put(20,22){\makebox(0,0){$\mathrm{Mod}^*(S)$}}
\put(20,6){\makebox(0,0){$\mathrm{Aut}(AC(S))$}}
\put(58,13){\makebox(0,0){$\mathrm{Aut}(C(S))$}}
\put(20,19){\vector(0,-1){10}}
\put(30,21){\vector(3,-1){17}}
\put(32,6){\vector(3,1){15}}
\put(23,14){\makebox(0,0){$\Phi$}}
\put(40,11){\makebox(0,0){$\Psi$}}
\put(40,21){\makebox(0,0){$\Theta$}}
\end{picture}
\end{center}

Let $\varphi$ be an automorphism of $AC(S)$ such that $\varphi$
fixes each curve. We claim that $\varphi$ is the identity. This will show that
$\Psi$ is one-to-one, concluding the proof that the natural map $\Phi$
is onto.

Let $a$ be an arc on $S$. It suffices to show that $\varphi (a)=a$.

Suppose first that $a$ joins a puncture $P$ to itself and that
both boundary components of a regular neighborhood of $a\cup
\{P\}$ are nontrivial curves. One can easily construct a (finite)
set $\Sigma$ of curves on $S$ such that each element of $\Sigma$
is disjoint from $a$, and that $a$ is the only arc disjoint from
all elements of $\Sigma$. Since $\varphi$ fixes each element of
$\Sigma$ and $\varphi (a)$ is an arc disjoint from all elements of
$\Sigma$, we conclude that $\varphi (a)=a$.

Suppose next that $a$ joins two distinct punctures, say $P$ and $Q$.
One can construct a (finite) set $\Sigma$ consisting of curves and
arcs, as in the previous paragraph, on $S$ such that
each element of $\Sigma$ is disjoint from $a$, and that
$a$ is the only arc disjoint from all elements of $\Sigma$.
Since $\varphi$ fixes each element of $\Sigma$ and $\varphi (a)$ is an
arc disjoint from all elements of $\Sigma$, we get that $\varphi (a)=a$.

Suppose finally that $a$ joins a puncture $P$ to itself, but one of the
boundary components of a regular neighborhood of $a\cup \{P\}$ is trivial,
i.e. bounds a disc with one puncture, say $Q$. Let $D$ be the disk with two
punctures $P$ and $Q$ such that $a$ lies on $D$.
Let $a'$ be the arc joining $P$ and $Q$ on $D$, which is disjoint from $a$.
Let $\Sigma$ be a finite set of arcs and curves on $S$
such that the elements of $\Sigma$ are fixed by $\varphi$, and that
$a$ and $a'$ are the only arcs disjoint from the elements of $\Sigma$.
Since $\varphi$ fixes each
element of $\Sigma\cup \{ a'\}$, we conclude that $\varphi (a)=a$.

This completes the second proof of Theorem~\ref{th:auto}.
\medskip

We conclude this section with a few words on the case of surfaces excluded by the
hypothesis of Theorem~\ref{th:auto}.

In the case of a sphere with one puncture, $AC(S)$ is empty. In the case of a
sphere with two punctures, $AC(S)=A(S)$ is a single point, and in the
case of a sphere with three punctures, $AC(S)=A(S)$ is a finite complex
(see Figure~\ref{pantalon}). In these two special cases,
by Theorem~\ref{thm:IM}
the natural homomorphism from $\mathrm{Mod}^*(S_{g,n})$ to $\mathrm{Aut}(AC(S_{g,n}))$
is surjective and its kernel is $\mathbb{Z}_2=\mathbb{Z}/2\mathbb{Z}$, which is the center of
$\mathrm{Mod}^*(S_{g,n})$. Finally, in the case where $S$ is a torus with one
puncture, the arguments given in the first proof of  Theorem \ref {th:auto}
and the fact that $\mathrm{Mod}^*(S_{g,n})\to \mathrm{Aut}(A(S_{g,n}))$ is
surjective with kernel $\mathbb{Z}_2$ (see
Theorem~\ref{thm:IM}) show that the natural homomorphism
$ \mathrm{Aut}(AC(S_{g,n})) \to \mathrm{Aut}(A(S_{g,n}))$ is an isomorphism.
This shows that we have an isomorphism
$\mathrm{Mod}^*(S_{g,n})/\mathbb{Z}_2\simeq \mathrm{Aut}(AC(S))$.

\begin{figure}[!hbp]
\centering
\includegraphics[width=0.9\linewidth]{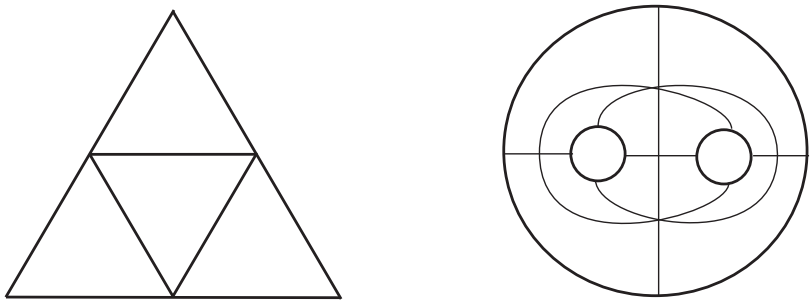}
\caption{\small{The figure on the left hand side represents the arc complex of
the sphere with three punctures.
It is a finite simplicial complex, having six vertices, nine edges and four
two-cells. The vertices of this complex are
the isotopy classes of the arcs represented in the right hand side figure,
which represents a three-punctured sphere in
which the punctures have been replaced
by circles, for more visibility. }}
\label{pantalon}
\end{figure}

\section{Proof of Theorem \ref{th:distinct}}

In this section, we consider the following classification of the vertices of
$AC(S)$ into five types.  A representative on $S$ of a vertex is of one of the
following types:
\begin{enumerate}
 \item \label{c1} a separating closed curve;
 \item  \label{c2} an arc connecting a puncture to itself that is
  separating the surface;
 \item  \label{c3} a nonseparating closed curve;
 \item  \label{c4} an arc connecting a puncture to itself that is
  nonseparating the surface;
\item  \label{c5} an arc connecting two distinct punctures.
\end{enumerate}

The stars and the
links of vertices will be used to distinguish combinatorially these types of
vertices from each other. For the convenience of the reader, we recall the
definition of the star and link of a vertex in a simplicial complex, which will
be useful below. Given a vertex $v$ in an abstract simplicial
complex $E$, its \textit{star}, $\mathrm{St}(v)$, is the subcomplex of $E$ whose
simplices are the simplices of $E$ that are of the following types: (1)
the simplices of $E$ that contain $v$, and (2) the faces of such simplices.
The \textit{link} of $v$,  $\mathrm{Lk}(v)$, is  then the subcomplex of $E$
whose simplices are
those simplices of  $\mathrm{St}(v)$ whose intersection with $v$ is empty.

We now prove Theorem~\ref{th:distinct}.

In the above classification of vertices of $AC(S)$, Cases  (\ref{c1}) and
(\ref{c2}) can be distinguished from the three others by the fact that
the link of a vertex in any one of these two cases is the join of two nonempty
subcomplexes.
Equivalently, the dual link, in each of these two cases, is the disjoint union of
two nonempty sets.
Here, by the {\it dual link} $L^*$ of a link $L$, we mean an abstract graph whose
vertices are the vertices of $L$,
and two vertices of $L^*$ are connected by an edge if and only if these vertices
are not connected by an edge when they are considered as vertices in $L$. In cases
(\ref{c3})-(\ref{c5}), the dual link is connected.

Case  (\ref{c1}) can be distinguished from Case (\ref{c2}) by the fact that
the highest dimension of a maximal simplex of a vertex in Case
(\ref{c2}) is bigger
than the  corresponding dimension  for Case  (\ref{c1}). (In fact, this
dimension
is $6g+3n-8$ in Case (\ref{c1}) and $6g+3n-7$ in Case (\ref{c2}).)

Thus, each of Cases  (\ref{c1}) and (\ref{c2}) is completely distinguished from all the
other cases.

Now, we deal with the remaining three cases. Case (\ref{c3}) is
distinguished from the other two by looking at the highest dimension
of maximal simplices containing such a vertex. The highest dimension of
any maximal simplex containing a vertex in Case (\ref{c3}) is $6g+3n-8$,
whereas in the other two cases this dimension is $6g+3n-7$.

It remains to distinguish Case (\ref{c4}) from Case (\ref{c5}). This is done by the
following property. Consider on the closed  surface $\overline{S}$ a closed curve $Z$ which
is the boundary of a regular neighborhood of an arc joining the two distinguished points,
representing a vertex in Case (\ref{c5}), which we call $a$. The vertex $z$ represented
by the curve $Z$ is in $\mathrm{Lk}(a)$ and it
has the following property: $z$ is represented by a
separating curve and $\mathrm{St}(z)\subset  \mathrm{St}(a)$.
Note that we already distinguished the vertices in $AC(S)$ represented by arcs
from those represented by closed curves by combinatorial information at the links of
these vertices; likewise, we already distinguished combinatorially
the vertices represented by separating closed curves from those represented by non-separating
curves. Therefore, the property stated is a combinatorial property of the link of $a$.
Now in the link of a vertex $a$ of Type  (\ref{c4}), there is no vertex $x$ represented by a
separating curve $X$ satisfying   $\mathrm{St}(x)\subset  \mathrm{St}(a)$.
Indeed, if $X$ is any separating curve on $S$ disjoint from an arc $A$
representing the vertex $a$, then  $A$ is in one of the two connected
 components of $S\setminus X$. This component has positive genus
 (since $A$ is also non-separating in that component), therefore we can find a closed
 curve $Y$ in $S\setminus X$ that intersects $A$ in exactly one point. Now the vertex
 $y$ represented by $Y$ is in $\mathrm{St}(x)$, but not in $\mathrm{St}(a)$,
 showing that $\mathrm{St}(x)$ is not a subset of $\mathrm{St}(a)$.
This completes the proof of the theorem.

\medskip

 Theorem \ref{th:distinct} shows that not only a full automorphism of the
 complex $AC(S)$
 recognizes the topological types of the vertices, as is implied by
 Theorem~\ref{th:auto},
 but that also that local combinatorial data contained in the links of each
 vertex are sufficient to characterize these types.

\section{Proof of Theorem~\ref{th:quasi}}
We equip  $C(S)$ and $AC(S)$ with the simplicial metrics recalled in the introduction.
Let $C^0(S)$ be the set of vertices of $C(S)$. Since the dimension of simplices in
the complex $AC(S)$ is uniformly
bounded, the inclusion map $C^0(S)\to C(S)$ is a quasi-isometry.
Thus, in order to show that the natural map $C(S)\to AC(S)$ is a
quasi-isometry, it suffices to show that the natural map
$C^0(S)\to AC(S)$ is a quasi-isometry. Thus, it suffices to reason
on distances between vertices of $C(S)$.

 We let $C^0(S)\to AC(S)$ be the natural inclusion map,
  and we let $d_C$ (respectively $d_{AC}$) denote the distance function in $C(S)$
  (respectively $AC(S)$). We identify $C^0(S)$ with its image in $AC(S)$.

  To see that every point in $ AC(S)$ is at distance at most one from a point
  in $C(S)$, we consider, for each essential arc $A$ on $S$, one of
  the boundary
  components of a regular neighborhood of $A$. Since the surface $S$ is not a
  sphere with at most three
  punctures, we can choose this boundary component to be essential. The isotopy
  class of such an essential simple closed curve is a vertex in $AC(S)$ which is at
  distance one from the vertex represented by the arc $A$.

  For the proof of the first claim in Theorem~\ref{th:quasi}, we
  prove that for every pair of vertices $x$ and $y$ in $C(S)$, we have
  \begin{eqnarray} \label{ineq:1}
    \frac{1}{2} \ d_C(x,y) \leq  d_{AC}(x,y)\leq d_C(x,y)
  \end{eqnarray}
 if $2g+n\geq 5$,
 \begin{eqnarray} \label{ineq:2}
   d_C(x,y)\leq d_{AC}(x,y) \leq  d_C(x,y) +2
 \end{eqnarray}
   if $(g,n)=(1,1)$, and
  \begin{eqnarray} \label{ineq:3}
   \frac{1}{2}\ d_C(x,y) \leq   d_{AC}(x,y)\leq d_C(x,y)+2
  \end{eqnarray}
 if $(g,n)=(0,4)$.

This will complete the proof of Theorem~\ref{th:quasi}.

 If $x=y$ then clearly all inequalities hold. So we assume that $x$ and $y$ are
 distinct.

 Suppose first that $2g+n\geq 5$.
 The second inequality in~(\ref{ineq:1}) is clear from the definitions.
 We prove the first one.

 Suppose that $ d_{AC}(x,y)=L$.
 Consider a simplicial path $p$ of length $L$ in $AC(S)$ joining $x$
 and $y$.
 We show that we can find a simplicial path $p'$ of length $\leq 2L$ in $AC(S)$
 joining these two vertices such that all the vertices of $p'$ are
 in $C(S)$.
 The construction of the path $p'$ is by induction. We describe the first step,
 and the induction step is similar.

  \medskip

\begin{figure}[!hbp]
\psfrag{1}{(i)}
\psfrag{2}{(ii)}
\psfrag{3}{(iii)}
\psfrag{4}{(iv)}
\psfrag{5}{(v)}
\psfrag{6}{(vi)}
\psfrag{7}{(vii)}
\psfrag{8}{(viii)}
\psfrag{9}{(ix)}
\centering
\includegraphics[width=0.9\linewidth]{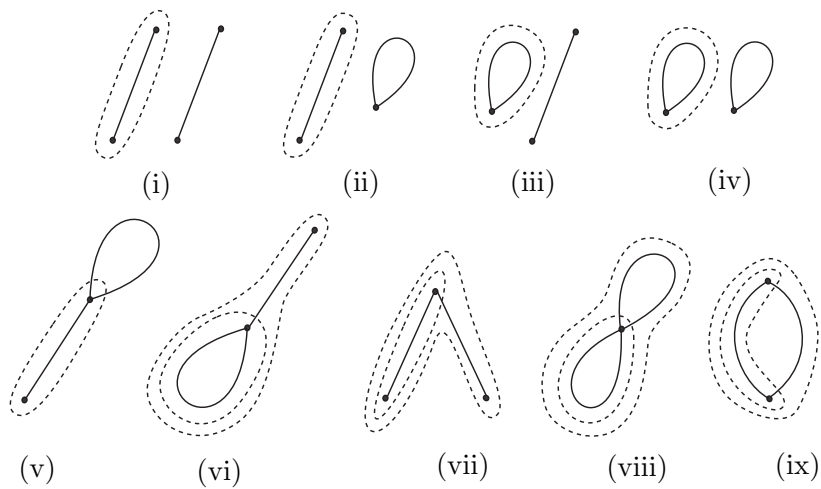}
\caption{\small{These nine cases, used in the proof of Theorem \ref{th:quasi},
are the various possibilities for an ordered pair of consecutive vertices
of $AC(S)$
that represent by arcs on $S$. The dashed lines are closed curves used to
replace the arcs.}}
\label{arcs}
\end{figure}

 There are three cases.

 \medskip

\noindent  {\it Case 1}. The first edge of the simplicial path $p$
 joins the vertex $x$ to a vertex $z$
 represented by a simple closed curve in $S$. In this case, we do not change this
 part of the path $p$, and we go to the next step, that is, we examine the part of
 the path $p$ that starts at the second vertex, $z$.

 \medskip

\noindent  {\it Case 2}. The first edge of the path $p$ joins the vertex
 $x$ to a vertex $a$
 represented by an arc $A$ in $S$, and the second edge joins the vertex $a$
 to a vertex $z$ represented by
 a curve. Note that $x$ is not connected to $z$
 by an edge. Note also that a tubular neighborhood $N$ of $A$ is either
 a disk with two punctures or
 an annulus with one puncture. Since $2g+n\geq 5$, one of the boundary
 components, say $V$, of $N$
 is a simple closed curve not bounding a disk with at most one puncture.
 In this case,
 let $v$ be the isotopy class of $V$. Note that $v$ is a vertex of $C(S)$
 which is at distance
 one from $x$ and from $z$. We replace the path
 $x a z$ by the path $xvz$.
 Such a step does not change the length of the path joining $x$ to $y$.

 \medskip

  \noindent  {\it Case 3}.  The first edge of the path $p$ joins the vertex
  $x$
  to a vertex $a$ represented by an arc $A$ on $S$, and the second edge joins  the
  vertex $a$ to a vertex $b$
  that is also represented by an arc $B$. We may assume that $A$ and $B$ are
  disjoint.
  In this case, we distinguish nine subcases, depending on the relative position of
  $A$ and $B$ on the
  surface $S$. These cases are represented in Figure~\ref{arcs},
  (i) to (ix). In each case, we replace the path $x a b$ by either a path of
  the same length,
  $x z b$, or a path of length three, $x z v b$. Here, $z$ is the
  isotopy class of a
  nontrivial boundary component of $A$, and $v$ is the isotopy class of a
  nontrivial boundary component of $A\cup B$. The curves
  are represented in dashed lines in Figure~\ref{arcs}, and one can check in each
  case that
  either a vertex $z$ or a pair of adjacent vertices $z$ and $v$ with the
  appropriate
  requirements exists. Note that this also uses fact that the surface $S$ is not a
  sphere with at most three punctures.

  Thus we get a simplicial path $p'$ between $x$ and $y$ of length $\leq 2L$
  such that all all vertices are in $C(S)$.
  This concludes the proof of~(\ref{ineq:1}).

  Suppose now that $(g,n)=(1,1)$. For each vertex $z$ in $AC(S)$
  represented by a curve, there is a unique vertex $z'$ represented by an arc
  such that $zz'$ is an edge in $AC(S)$.
  Conversely, for each vertex $z'$ in $AC(S)$ represented by an arc
  there is a unique vertex $z$ represented by a curve such that
  $z'z$ is an edge in $AC(S)$.
  From this, it is easy to see that for any two vertices $z$ and $v$ represented by
  curves
  \begin{enumerate}
    \item [(i)]
    there are no simplicial paths of form $zv$, or $zav$, or $azb$ in $AC(S)$, where
    $a$ and $b$ are represented by arcs, and
    \item [(ii)] if there is a simplicial path $zabv$ in $AC(S)$ then
  $d_C(z,v)=1$.
  \end{enumerate}

  Suppose that $d_C(x,y)=L$. Since $x$ and $y$ are distinct,
  we have $L\geq 3$ by~(i). If $ xx_1x_2\cdots x_{L-1}y $
  is a path of length $L$ in $C(S)$, so that adjacent vertices have representatives
  intersecting only once, then
  $ x x' x'_1 x'_2 \cdots x'_{L-1} y' y$ is a path of length $L+2$ in $AC(S)$,
  showing that $d_{AC}(x,y)\leq d_C(x,y)+2$.

  Suppose now that $d_{AC}(x,y)=L$. Let $p$ be a simplicial path of length $L$
  in $AC(S)$ joining $x$ and $y$. By~(i),
  the vertices of $p$ other than $x$ and $y$ are all
  represented by arcs. Thus the
  path $p$ must be of the form $xx'x'_1\cdots x'_{L-3}y'y$. Then by~(ii), the path
  $xx_1x_2\cdots x_{L-3}y$ is of length $L-2$ in $C(S)$, showing that
  $d_C(x,y)\leq d_{AC}(x,y)-2$. This finishes the proof of~(\ref{ineq:2}).

  Suppose finally that $(g,n)=(0,4)$. For each vertex $a$ in $AC(S)$
  represented by an arc, there is a unique vertex $v_a$ represented by a curve
  such that $av_a$ is an edge in $AC(S)$.
  For each vertex $z$ in $AC(S)$ represented by a curve,
  there are two vertices $a$ and $b$ connected to
  $z$ by an edge such that representatives of each of $a$ and $b$ joins
  different punctures. Let $z'$ be any one of them.
  Note that if $z$ and $v$ are connected by an edge in $C(S)$, so that
  they have representatives intersecting twice, then
  $zz'v'v$ is a simplicial path in $AC(S)$ for any choice of $z'$ and $v'$.
  It is now easy to see that for any two vertices $z$ and $v$ represented by
  curves
  \begin{enumerate}
    \item [(iii)] there are no simplicial paths of the form $zv$, or $zav$ in $AC(S)$,
      where $a$ is
      represented by an arc, and
    \item [(iv)] if there is a simplicial path $zabv$ in $AC(S)$ then
      $d_C(z,v)$ is equal to $1$ or $2$.
  \end{enumerate}

  Suppose that $d_C(x,y)=L$. Since $x$ and $y$ are distinct,
  we have $L\geq 3$.
  Let $ xx_1x_2\cdots x_{L-1}y $ be a path of length
  $L$ in $C(S)$, so that adjacent vertices have representatives
  intersecting twice. Then $ x x' x'_1 x'_2 \cdots x'_{L-1} y' y$
  is a path of length $L+2$ in $AC(S)$, showing that $d_{AC}(x,y)\leq d_C(x,y)+2$.

  Suppose now that $d_{AC}(x,y)=L$. Let $p$ is a simplicial path of length $L$
  in $AC(S)$ joining $x$ to $y$. By~(iii), we have $L\geq 3$. If there is a part of $p$
  of the form $azb$, then by changing $z$ with $z'$ for a suitable choice,
  we can assume that all vertices of $p$ other than $x$ and $y$ are represented by arcs.

  Thus the path $p$ must be of the form $xa_1a_2\cdots a_{L-1}y$, where each $a_i$
  is represented by an arc. If $v_i$ denotes
  the vertex $v_{a_i}$, we get a sequence of vertices $x,v_2,v_3,\cdots ,v_{L-2},y$
  in $C(S)$. Since the distance between adjacent vertices in this sequence
  is at most $2$ by~(iv), we get that $d_C(x,y)\leq 2(L-2)$, concluding that
  $\frac{1}{2}\ d_C(x,y)+2\leq d_{AC}(x,y)$.

  This finishes the proof of~(\ref{ineq:3}), concluding the proof of
  Theorem~\ref{th:quasi}.

  \medskip

  \noindent {\it Acknowledgments} The authors thank A. Hatcher and S. Schleimer for
  pointing out the references \cite{Hatcher} and \cite{Schleimer} after a first
  version of this paper was written.

\end{document}